\documentclass[11pt]{article}%
\usepackage{amsmath}
\usepackage{amsfonts}
\usepackage{amssymb}
\usepackage{graphicx}%
\newtheorem{theorem}{Theorem}

\newtheorem{conjecture}[theorem]{Conjecture}

\newtheorem{lemma}[theorem]{Lemma}

\newenvironment{proof}[1][Proof]{\noindent\textbf{#1.} }{\ \rule{0.5em}{0.5em}}
\graphicspath{    {converted_graphics/}    {/}}
\begin{document}

\begin{center}
{\LARGE Bounded Orbits of Quadratic Collatz-type Recursions}

\medskip

H. SEDAGHAT \footnote{Email: hsedagha@vcu.edu}

\end{center}

\bigskip

\begin{abstract}
We characterize all bounded orbits of two similar Collatz-type quadratic
mappings of the set of non-negative integers. In one case, where cycles of all
possible lengths may occur, an orbit is bounded if and only if it reaches a
cycle. For the other map we prove that every bounded orbit must reach 0 (in
particular, there are no cycles).

\end{abstract}

\medskip

\section{Introduction}

Let $\mathbb{N}$\ be the set of all positive integers and $\mathbb{N}%
_{0}=\{0,1,2,3,\ldots\}$ be the set of all non-negative integers. Consider the
function $Q:\mathbb{N}_{0}\rightarrow\mathbb{N}_{0}$ that is defined as:%
\begin{equation}
Q(n)=\left\{
\begin{array}
[c]{l}%
n/2\quad\text{if }n\text{ is even}\\
\binom{n}{2}\text{\quad if }n\text{ is odd}%
\end{array}
\right.  \label{dc2}%
\end{equation}
where
\[
\binom{n}{2}=\frac{n(n-1)}{2}%
\]

\noindent is the binomial coefficient. We call the function $Q$ in (\ref{dc2})
the \textit{divide-or-choose-2 rule}.

Note that for odd $n$ the function $Q$ is a \textit{multiplicative} version of
the Collatz-type map%
\[
F(n)=\left\{
\begin{array}
[c]{l}%
n/2\qquad\text{if }n\text{ is even}\\
\frac{3n-1}{2}\text{\quad\ if }n\text{ is odd}%
\end{array}
\right.
\]
in the following sense: Since $n-1$ is even for odd $n$ we may write%
\[
\binom{n}{2}=n\left(  \frac{n-1}{2}\right)  ,\qquad\frac{3n-1}{2}=n+\frac
{n-1}{2}%
\]

Equivalently, $F$ is the \textit{additive} or \textquotedblleft linear"
version of $Q$. It is a variant of the better known (compressed) Collatz
function%
\[
T(n)=\left\{
\begin{array}
[c]{l}%
n/2\qquad\text{if }n\text{ is even}\\
\frac{3n+1}{2}\text{\quad\ if }n\text{ is odd}%
\end{array}
\right.
\]

While all orbits of $T$ are conjectured to reach the base cycle
$\{1,2,1,2,\ldots\}$ from any initial value $n\in\mathbb{N}$ the variant $F$
generates nontrivial cycles \cite{LG}.

\medskip

The variant of $Q$ that represents a multiplicative version of the Collatz
function $T$ is%
\[
\left\{
\begin{array}
[c]{l}%
\ n/2\quad\quad\quad\text{if }n\text{ is even}\\
\binom{n+1}{2}\text{\quad\ \ if }n\text{ is odd}%
\end{array}
\right.
\]

The orbits that are generated by this function are qualitatively similar to
those generated by $Q$ so we need not consider this map in detail. Instead, we
study the following variant of $Q$
\[
S(n)=\left\{
\begin{array}
[c]{l}%
\ n/2\quad\ \text{if }n\text{ is even}\\
\frac{n^{2}-1}{4}\text{\quad if }n\text{ is odd}%
\end{array}
\right.
\]

The function $S$ has a symmetric expression in the sense that
\[
\frac{n^{2}-1}{4}=\left(  \frac{n-1}{2}\right)  \left(  \frac{n+1}{2}\right)
\]
so we call it the \textit{symmetric rule}.

\medskip

Our goal in this paper is two-fold: We show that the orbits of $Q$ may reach
cycles of all possible lengths. All cycles contain an odd number of type
$2^{m}+1$ for some $m\in\mathbb{N}_{0}$ and an orbit of $Q$ that does not
reach a cycle is shown to be unbounded. This result characterizes all bounded
orbits that may be generated by iterating $Q$. We similarly show that the
orbits of $S$ must either reach 0 or be unbounded and each orbit that reaches
0 contains a number of type $2^{m}\pm1$ for some $m\in\mathbb{N}_{0}$. The
existence of an unbounded orbit (i.e. an orbit that does not satisfy the
preceding conditions and \textquotedblleft escapes to infinity") for either
$Q$ or $S$ is a separate problem and not discussed in this paper.

\medskip

A note about the domains of the above quadratic maps: Both may be extended to
the set $\mathbb{Z}$ of all integers rather trivially in the sense that every
orbit with a negative initial value enters $\mathbb{N}_{0}$ after a finite
number of steps and stays there since $\mathbb{N}_{0}$ is invariant under both
mappings. We do not consider extensions to $\mathbb{Z}$ here.

\section{The divide-or-choose-2 rule}

If $x_{0}\in\mathbb{N}_{0}$ then the numbers%
\[
x_{0},Q(x_{0}),Q(Q(x_{0})),\ldots
\]
constitute an \textit{orbit} or \textit{trajectory} of the quadratic recursion%
\begin{equation}
x_{n+1}=Q(x_{n}) \label{q1}%
\end{equation}
in $\mathbb{N}_{0}$. If $x_{m}=x_{0}$ for some $m\in\mathbb{N}$ then the
numbers $x_{0},x_{1},\ldots,x_{m-1}$ repeat so we have a \textit{periodic
orbit with period} $m$ or equivalently, an $m$-\textit{cycle}, i.e. a cycle of
length $m$. A number $x_{0}$ that lies on a cycle is called a \textit{periodic
point} of $Q$. If $m=1$ then the cycle is often called a \textit{fixed point}
of $Q$.

\medskip

The divide-or-choose-2 rule works as follows: If the value of $x_{0}$ is even
then we \textit{divide }it by 2 to get the next term:
\[
x_{1}=Q(x_{0})=\frac{x_{0}}{2}%
\]

On the other hand, if $x_{0}$ has odd value then the next term is $x_{0}$
\textit{choose }2:%
\[
x_{1}=\binom{x_{0}}{2}=x_{0}\left(  \frac{x_{0}-1}{2}\right)
\]

Let $m$ be a positive integer and consider $x_{0}=2^{m}.$ Then%
\[
x_{1}=2^{m-1},\ x_{2}=2^{m-2},\ \ldots,x_{m-1}=2,\ x_{m}=1
\]

At this point, since 1 is odd,%
\[
x_{m+1}=0
\]
and with 0 being even, $x_{n}=0$ for all $n>m$. The repeating number 0 is a
fixed point of (\ref{q1}) since as an even number, $Q(0)=0/2=0.$ One more
fixed point is 3, since%
\[
Q(3)=\binom{3}{2}=3.
\]

By examining the expressions for odd and even numbers it is easy to see that
$Q$ has no other fixed points.

\medskip

More generally, if $x_{0}=2^{m}k$ is an arbitrary even number where
$m\in\mathbb{N}_{0}$ and $k$ is odd then%
\begin{equation}
x_{1}=2^{m-1}k,\ x_{2}=2^{m-2}k,\ \ldots,\ x_{m}=k \label{evo}%
\end{equation}

Also every odd number larger than 1 can be written as $2^{m}k+1$ where
$m\in\mathbb{N}_{0}$ and $k$ is odd. Note that
\begin{equation}
Q(2^{m}k+1)=(2^{m}k+1)\left(  \frac{2^{m}k+1-1}{2}\right)  =2^{m-1}k(2^{m}k+1)
\label{odo}%
\end{equation}
so if $x_{0}=2^{m}k+1$ then%
\[
x_{1}=2^{m-1}kx_{0}%
\]

If $m>1$ then $x_{2}$ is just half of $x_{1}$ so that $x_{2}=2^{m-2}kx_{0}$.
By induction%
\begin{equation}
x_{m}=kx_{0} \label{odom}%
\end{equation}

This observation has an interesting consequence: if $k=1$ then $x_{m}=x_{0}$
and we obtain an $m$-cycle.

Notice that since $m$ is any positive integer in the above argument, we have
proved the following.

\medskip

\begin{lemma}
For every positive integer $m$ the recursion (\ref{q1}) has an $m$-cycle given
by the numbers (in the order shown):%
\begin{equation}
2^{m}+1\rightarrow2^{m-1}(2^{m}+1)\rightarrow2^{m-2}(2^{m}+1)\rightarrow
\ldots\rightarrow2(2^{m}+1)\rightarrow2^{m}+1 \label{cyc}%
\end{equation}

\end{lemma}

\medskip

The exceptional value $x_{0}=1$ is mapped to 0 and the 1-cycle 0 is reached.
We can also infer from (\ref{evo}) that each number of type $2^{j}(2^{m}+1)$
reaches the $m$-cycle of the above lemma in $j$ steps for every $j\in
\mathbb{N}$. But these are not the only numbers that may reach cycles. Suppose
that $k>1$ in (\ref{odo}), say, $k=k_{0}\geq3$. Starting with $x_{0}=2^{m_{0}%
}k_{0}+1$, by (\ref{odom}) the orbit reaches%
\[
x_{m_{0}}=k_{0}x_{0}\geq3x_{0}%
\]

Therefore, on the down-swing the orbit does not reach $x_{0}$ to form a cycle
but instead, it reaches a larger odd number $k_{0}x_{0}$. Set $k_{0}%
x_{0}=2^{m_{1}}k_{1}+1$ where $m_{1}$ is a positive integer and $k_{1}$ is
odd. Calculating as before,%
\[
x_{m_{0}+m_{1}}=k_{1}x_{m_{0}}=k_{1}k_{0}x_{0}%
\]

We have two possible cases: $k_{1}=1$ in which case $x_{m_{0}+m_{1}}=x_{m_{0}%
}$ and the orbit has reached an $m_{1}$-cycle. Otherwise, \thinspace$k_{1}>1$
and
\[
x_{m_{0}+m_{1}}=k_{1}x_{m_{0}}\geq3x_{m_{0}}\geq3^{2}x_{0}%
\]

This process may be repeated by setting $k_{1}x_{m_{0}}=2^{m_{2}}k_{2}+1$ as
long as the coefficients $k_{2}$ etc remain larger than 1. We obtain the
general expression%
\[
x_{m_{0}+m_{1}+\cdots+m_{p}}=k_{p}x_{p-1}=k_{p}\cdots k_{1}k_{0}x_{0}%
\geq3^{p+1}x_{0}%
\]
where $p$ is a positive integer. If $k_{p}>1$ for all $p$ then this process
generates ever larger values that grow infinitely large. Therefore, either the
orbit reaches a cycle or it is unbounded. Since orbits starting with an even
number always reach an odd number, the above argument proves the following
characterization of the bounded orbits of (\ref{q1}).

\medskip

\begin{theorem}
The recursion (\ref{q1}) has cycles of type (\ref{cyc}) of all possible
lengths. Every orbit of (\ref{q1}) either reaches such a cycle or it is unbounded.
\end{theorem}

We emphasize that the proof of the above theorem does not establish the
existence of orbits that never reach cycles so it does not imply that
(\ref{q1}) has any unbounded orbits, i.e. orbits that \textquotedblleft escape
to infinity". On the other hand, the theorem gives a complete characterization
of all bounded orbits; this much is not known for the classic recursion of Collatz.

\medskip

It is expected that orbits are generally unbounded given the quadratic growth
rate in the odd case and because at each iteration, there is a 50 percent
chance that $x_{n}$ is divided by 2, and if not then it is \textit{squared}
(essentially). 

\begin{conjecture}
The recursion (\ref{q1}) has an unbounded orbit.
\end{conjecture}

If there is an unbounded orbit then there are infinitely many, for if an (odd)
number $x_{0}=n$ leads to an unbounded orbit then so do the numbers $2^{m}n$
for all $m\in\mathbb{N}$. The following makes a more specific proposal:

\begin{conjecture}
\textit{Orbits containing an odd number of type }$2^{m}-1$\textit{ are
unbounded} for all $m\geq3$.
\end{conjecture}

It is worth a mention that numerical simulations do not prove this statement
and they may even lead to false conclusions on digital computers. The reason
seems to be that for large $k$ or $m$, the crucial distinction between 
$2^{m}k$ and $2^{m}k\pm1$ is typically missed, causing the software to
produce a cycle where none exists.

\medskip

Next, note that for every $m\in\mathbb{N}$ repeated applications of $Q$ to a
number of type $2^{m}k$ where $k$ is odd leads to $k$ through a monotonically
decreasing chain $m$ steps long. In particular, \textit{decreasing} chains of
arbitrary length are possible. The following result shows not only that orbits
with \textit{increasing} chains of arbitrary length occur but also gives a
type of number that leads to them.

\begin{theorem}
Let $x_{0}=2^{m}+3$ where $m\geq2$. Then for $j=1,2,\ldots,m-1$%
\begin{align}
x_{j}  & =2^{m-j}(x_{0}+2)(x_{1}+2)\cdots(x_{j-1}+2)+3\label{t1}\\
x_{m}  & =x_{m-1}[(x_{0}+2)(x_{1}+2)\cdots(x_{m-2}+2)+1]\label{t2}%
\end{align}

In particular, $x_{j}$ is odd for each $j$ and $x_{0}<x_{1}<\cdots<x_{m-1}$ is
an increasing chain reaching the even number $x_{m}$.
\end{theorem}

\begin{proof}
We use induction. Since $x_{0}$ is odd,%
\[
x_{1}=Q(x_{0})=x_{0}\left(  \frac{2^{m}+2}{2}\right)  =(2^{m}+3)(2^{m-1}+1)
\]

Multiplying out the last expression and collecting terms%
\[
x_{1}=(2^{m}+3)2^{m-1}+2^{m}+3=2^{m-1}(2^{m}+3+2)+3=2^{m-1}(x_{0}+2)+3
\]

This proves (\ref{t1}) for $j=1$. If $m>2$ and (\ref{t1}) holds for $k<m-1$
then
\begin{align*}
x_{k+1}  & =x_{k}\left(  \frac{x_{k}-1}{2}\right)  \\
& =x_{k}[2^{m-k-1}(x_{0}+2)(x_{1}+2)\cdots(x_{k-1}+2)+1]\\
& =x_{k}2^{m-k-1}(x_{0}+2)(x_{1}+2)\cdots(x_{k-1}+2)+2^{m-k}(x_{0}%
+2)(x_{1}+2)\cdots(x_{k-1}+2)+3\\
& =2^{m-(k+1)}(x_{0}+2)(x_{1}+2)\cdots(x_{k-1}+2)(x_{k}+2)+3
\end{align*}

It follows by induction that (\ref{t1}) is true as long as $j<m$. For $j=m-1$
the application of $Q$ to the odd number $x_{m}=x_{j+1}$ gives (\ref{t2}).
\end{proof}

\section{The symmetric rule}

For comparison, we now consider the function $S$ which defines the recursion%
\begin{equation}
x_{n+1}=S(x_{n}) \label{sq1}%
\end{equation}

Like $Q$, if $x_{0}=2^{m}$ for some positive integer $m$ then $x_{m}=1$ and
$x_{m+1}=0.$ We conclude that the orbit of $2^{m}$ reaches 0 in $m+1$ steps.
It follows that $x_{n}=0$ for all $n\geq m+1$. The next lemma extends this
observation to similar but odd initial values. The appearance of $\pm$
reflects the symmetry in $S$ that was lacking in $Q$.

\medskip

\begin{lemma}
\label{L2}For every positive integer $m$ if $x_{0}=2^{m}\pm1$ then $x_{n}=0$
for all $n\geq\binom{m+1}{2}+2.$
\end{lemma}

\begin{proof}
First, consider $x_{0}=2^{m}+1$. Then%
\[
x_{1}=\left(  \frac{2^{m}}{2}\right)  \left(  \frac{2^{m}+2}{2}\right)
=2^{m-1}(2^{m-1}+1)
\]

Further applying $S$ a total of $m-1$ times gives%
\[
x_{m}=2^{m-1}+1
\]

If $m=1$ then $x_{0}=2+1=3$ and $x_{1}=2$ from which we obtain $x_{2}=1$ and
$x_{3}=0$. If $m>1$ then repeating the above argument yields%
\[
x_{m+(m-1)}=2^{m-2}+1
\]

Continuing this way we obtain%
\[
x_{m+(m-1)+(m-2)}=2^{m-3}+1
\]

and so on until%
\[
x_{m+(m-1)+\cdots+2+1}=2^{0}+1=2
\]

Now two more applications of $S$ lead to the fixed value 0. The total number
of applications of $S$ is therefore,%
\[
m+(m-1)+\cdots+2+1+2=\frac{m(m+1)}{2}+2=\binom{m+1}{2}+2
\]

Thus for all $n$ larger than the above number, $x_{n}=0$. A similar argument
shows that if $x_{0}=2^{m}-1$ then%
\[
x_{m+(m-1)+\cdots+2+1}=2^{0}-1=0
\]

so $x_{n}=0$ for all $n\geq\binom{m+1}{2}$.
\end{proof}

\medskip

There are positive integers $x_{0}$, e.g., $x_{0}=2^{p}(2^{m}-1)$ that reach
$2^{m}\pm1$ after several iterations of $S$. The orbits of all such initial
values reach 0 in a finite number of steps. What happens if the numbers
$2^{m}\pm1$ are never reached from some initial value $x_{0}$? The following
answers this question.

\medskip

\begin{theorem}
Every orbit of the recursion (\ref{sq1}) either reaches zero or it is
unbounded, i.e. escapes to infinity.
\end{theorem}

\begin{proof}
We may start with an odd initial value, $x_{0}=2^{m_{0}}k_{0}-1$ where $k_{0}$
is odd. Note that%
\[
x_{1}=\left(  \frac{2^{m_{0}}k_{0}-2}{2}\right)  \left(  \frac{2^{m_{0}}k_{0}%
}{2}\right)  =2^{m_{0}-1}k_{0}(2^{m_{0}-1}k_{0}-1)
\]

It follows that%
\[
x_{m_{0}}=k_{0}(2^{m_{0}-1}k_{0}-1)
\]

If $k_{0}=1$ then Lemma \ref{L2} implies that the orbit reaches zero in a
finite number of steps. If $k_{0}\geq3$ then note that%
\[
x_{m_{0}}\geq3(2^{m_{0}-1}k_{0}-1)=2^{m_{0}}k_{0}+2^{m_{0}-1}k_{0}-3=2^{m_{0}%
}k_{0}+3(2^{m_{0}-1}-1)\geq2^{m_{0}}k_{0}%
\]

Therefore, if $k_{0}>1$ then for all $m_{0}\geq1$%
\[
x_{m_{0}}\geq x_{0}+1
\]

The odd number $x_{m_{0}}$ is the lowest point of the down-swing following
$x_{1}$ but its value exceeds the initial value $x_{0}$.

Next, let $m_{1}$ and $k_{1}$ be positive integers with $k_{1}$ odd such that
$x_{m_{0}}=2^{m_{1}}k_{1}-1$. Knowing what happens if $k_{1}=1$, assume that
$k_{1}\geq3$. Repeating the above calculation, we conclude that%
\[
x_{m_{0}+m_{1}}\geq x_{m_{0}}+1\geq x_{0}+2
\]

This process continues with $m_{2},k_{2}$ etc and is stopped only if $k_{i}=1$
for some positive integer $i$. Otherwise, for every $p\in\mathbb{N}$ we have%
\[
x_{m_{0}+m_{1}+\cdots+m_{p}}\geq x_{0}+p+1
\]
implying that the orbit is unbounded.
\end{proof}

\medskip

Like the earlier case of divide-or-choose-2 map, the above theorem does not
state that all orbits reach 0. So the existence of unbounded orbits is not
implied even though we expect that they do exist, pehaps abundantly. However,
the theorem does characterize all \textit{bounded} orbits of (\ref{sq1}): they
must all reach 0.

\section{Conclusion}

We discussed two quadratic maps of Collatz type and fully characterized their
bounded orbits. Proving the expected existence of unbounded orbits for these
maps is left as\ an open problem. Similar ideas and methods may extend to
similar types of quadratic maps and help us understand the underlying
complexity of these systems a little better.

\medskip

The maps $Q$ and $S$, as well as the Collatz map $T$ and all other similar
maps represent examples of \textit{bimodal systems} in the sense that each map
is divided into two parts: one part is defined on the set of all even integers
$\mathcal{E}$ and the other on the set of all odd integers $\mathcal{O}$.
Specifically, $Q$ consists of an even part $Q_{0}=n/2$ and an odd part
$Q_{1}(n)=\binom{n}{2}$. Importantly, neither of these maps is a self map of
its domain; i.e., $\mathcal{E}$ is not invariant under $Q_{0}$ and
$\mathcal{O}$ is not invariant under $Q_{1}$. These features characterize
bimodal systems, which are defined in a general way in \cite{HS}.

The basic properties of bimodal systems (more generally, \textit{polymodal
systems}) are discussed in \cite{HS} where it is also seen that economic and
social science models are often bimodal and therefore, provide a rich source
of applications for results on bimodal systems. As maps like $T$ and $Q$
illustrate, these systems are capable of generating nontrivial dynamics in the
form of orbits that repeatedly enter and exit their domains, namely,
$\mathcal{E}$ and $\mathcal{O}$.

With $Q$, all bounded orbits end in $m$-cycles for $m\geq2$, exhibiting
persistent oscillations between $\mathcal{E}$ and $\mathcal{O}$ with the same
being possibly true of unbounded orbits. Similarly, if Collatz's conjecture is
true then all orbits of $T$ eventually oscillate between $\mathcal{E}$ and
$\mathcal{O}$ as they alternate between 2 and 1. On the other hand, the
bounded orbits generated by $S$ always end in 0 and thus remain in
$\mathcal{E}$; they do not oscillate persistently between $\mathcal{E}$ and
$\mathcal{O}$. The unbounded orbits of $S$ may oscillate persistently between
$\mathcal{E}$ and $\mathcal{O}$.

\medskip



\begin{thebibliography}{9}                                                                                                %


\bibitem {LG}Lagarias, J.C., (Editor) \textit{The Ultimate Challenge: The 3x+1
Problem}, American Mathematical Society, Providence, 2010

\bibitem {HS}Sedaghat, H., \textit{Nonlinear Difference Equations: Theory with
Applications to Social Science Models}, Springer, New York, 2003
\end{thebibliography}
\end{document}